\documentclass{amsart}
\usepackage{ifthen,amsmath,amsfonts,amstext,amsthm,latexsym,amssymb}
\usepackage{mathrsfs}
\usepackage[english]{babel}
\usepackage[utf8]{inputenc}
\usepackage{enumerate}
\usepackage{url}
\usepackage{graphicx}
\input xy
\xyoption{all}
\usepackage[all]{xy}

\newcommand{\nwc}{\newcommand}
\newcommand{\cliff}{{\rm Cliff}} 

\nwc{\aaa}{\mathcal{F}}
\nwc{\aap}{\mathcal{F}_{P}}
\nwc{\ab}{\mathfrak{a}}
\nwc{\ap}{\mathfrak{a}_{P}}
\nwc{\cb}{\overline{C}}
\nwc{\ccc}{\mathcal{C}}
\nwc{\ch}{\widehat{C}}
\nwc{\cin}{\textbf{(v)}}
\nwc{\cl}{C'}
\nwc{\cp}{\mathcal{C}_{P}}
\nwc{\cpll}{\mathfrak{c}_{P'}}
\nwc{\cq}{\mathfrak{C}_{Q}}
\nwc{\ct}{\widetilde{C}}

\nwc{\dd}{\mathcal{L}}
\nwc{\ddd}{\mathfrak{d}}
\nwc{\ddl}{\mathcal{L}'}
\nwc{\dlp}{\delta_{P}}
\nwc{\doi}{\textbf{(ii)}}

\nwc{\enq}{\$\$}

\nwc{\fff}{\mathcal{F}}
\nwc{\ffp}{\mathcal{F}_{P}}
\nwc{\ffq}{\mathcal{F}_{Q}}
\nwc{\ffl}{\mathcal{F}'}

\nwc{\G}{\mathcal{G}}
\nwc{\gtl}{\tilde{g}}
\nwc{\guu}{g^{1}_{1}}
\nwc{\gud}{g^{1}_{2}}
\nwc{\gut}{g^{1}_{3}}
\nwc{\gon}{{\rm gon}}

\nwc{\howw}{H^0(\omega)} 
\nwc{\how}{H^0(\omega)} 
\nwc{\hua}{h^{1}(C,\aaa )}

\nwc{\kk}{{\rm K}}

\nwc{\llb}{\mathcal{L}}

\nwc{\mm}{\mathfrak{m}}
\nwc{\mmp}{\mathfrak{m}_{P}}
\nwc{\mpd}{\mathfrak{m}_{P}^{2}}
\nwc{\mmq}{\mathfrak{m}_{Q}}

\nwc{\nn}{\mathbb{N}}

\nwc{\ob}{\overline{\mathcal{O}}}
\nwc{\obp}{\overline{\mathcal{O}}_P}
\nwc{\obr}{\mathcal{O}^*}
\nwc{\och}{\mathcal{O}_{\hat{C}}}
\nwc{\ocut}{\oo _{C}\langle 1,t\rangle}
\nwc{\oh}{\hat{\mathcal{O}}}
\nwc{\ohp}{\hat{\mathcal{O}}_{P}}
\nwc{\ol}{\mathcal{O}'}
\nwc{\oma}{\Omega (\mathfrak{a})}
\nwc{\omo}{\Omega (\mathcal{O})}
\nwc{\oo}{\mathcal{O}}
\nwc{\op}{\mathcal{O}_P}
\nwc{\opc}{\mathcal{O}_{P,C}}
\nwc{\oph}{\widehat{\mathcal{O}}_{P}}
\nwc{\opl}{\mathcal{O}_{P}'}
\nwc{\oplc}{\mathcal{O}_{P,C}'}
\nwc{\opll}{\mathcal{O}_{P'}}
\nwc{\opt}{\overline{\mathcal{O}}_{P}}
\nwc{\optt}{{\mathcal{O}}_{\tilde{P}}}
\nwc{\oq}{\mathcal{O}_{Q}}
\nwc{\oqt}{\tilde{\mathcal{O}}_{Q}}
\nwc{\ot}{\tilde{\mathcal{O}}}
\nwc{\overop}{\bar{\oo}_{P}}
\nwc{\ocux}{\oo _{C}\langle 1,x\rangle}

\nwc{\pb}{\overline{P}}
\nwc{\pbb}{P^*}
\nwc{\pgmd}{\mathbb{P}^{g+2}}
\nwc{\pgmu}{\mathbb{P}^{g+1}}
\nwc{\ph}{\hat{P}}
\nwc{\pp}{\mathbb{P}}
\nwc{\prv}{\noindent\textbf{Proof}:}
\nwc{\pt}{\tilde{P}}
\nwc{\ptl}{\tilde{P}}
\nwc{\pum}{\mathbb{P}^{1}}

\nwc{\qh}{\hat{Q}}
\nwc{\qtl}{\tilde{Q}}
\nwc{\qua}{\textbf{(iv)}}
\newcommand{\sss}{{\rm S}}
\nwc{\rh}{\hat{R}}

\nwc{\sei}{\textbf{(vi)}}
\nwc{\sepp}{\beq\ast\ \ast\ \ast\enq}
\nwc{\ssp}{S_{P}}
\nwc{\sys}{\mathcal{L}}

\nwc{\tcp}{t_{C'}(P')}
\nwc{\tcq}{t_{C'}(Q')}
\nwc{\tre}{\textbf{(iii)}}

\nwc{\um}{\textbf{(i)}}

\nwc{\vl}{\mathcal{W}}
\nwc{\vlp}{\mathcal{W}_{P}}
\nwc{\vpb}{v_{\overline{P}}}
\nwc{\vtxp}{\widetilde{V}_{x,P}}
\nwc{\vv}{\mathcal{W}}
\nwc{\vxp}{V_{x,P}}
\nwc{\vzp}{V_{z,P}}

\nwc{\wh}{\hat{\omega}}
\nwc{\whp}{\hat{\omega}_{P}}
\nwc{\woch}{\omega\cdot\mathcal{O}_{\hat{C}}}
\nwc{\woh}{\omega\cdot\hat{\mathcal{O}}}
\nwc{\ww}{\omega}
\nwc{\wwb}{\omega^*}
\nwc{\wwh}{\widehat{\omega}}
\nwc{\wwhp}{\widehat{\omega}_P}
\nwc{\wwp}{\omega _{P}}
\nwc{\wwt}{\widetilde{\omega}}
\nwc{\wwtp}{\widetilde{\omega}_P}

\nwc{\zz}{\mathbb{Z}}

\newtheorem{coro}{Corollary}[section]
\newtheorem{defi}[coro]{Definition}

\newtheorem{lem}[coro]{Lemma}

\newtheorem{thm}[coro]{Theorem}

\newtheorem*{remark}{Theorem}

\begin{document}

\title{Max Noether Theorem for Singular Curves}

\author{Edson Martins Gagliardi}
\address{Departamento de Ciências Exatas, FACSAE, UFVJM/Mucuri, R. Cruzeiro - Jardim Sao Paulo, 39803-371 Teófilo Otoni - MG}
\email{edson.gagliardi@ufvjm.edu.br}

\author{Renato Vidal Martins}
\address{Departamento de Matem\'atica, ICEx, UFMG
Av. Ant\^onio Carlos 6627,
30123-970 Belo Horizonte MG, Brazil}
\email{renato@mat.ufmg.br}

\maketitle

\begin{abstract}
Max Noether's Theorem asserts that if $\ww$ is the dualizing sheaf of a nonsingular nonhyperelliptic projective curve, then the natural morphisms $\text{Sym}^nH^0(\omega)\to H^0(\omega^n)$ are surjective for all $n\geq 1$. The result was extended for Gorenstein curves by many different authors in distinct ways. More recently, it was proved for curves  with projectively normal canonical models, and curves whose non-Gorenstein points are bibranch at worse. Based on those works, we address the combinatorics of the general case and extend the result for any integral curve. 
\end{abstract}

\section*{Introduction}

In late century XIX, Max Noether proved in \cite{N} a celebrated result establishing that every canonical curve is projectively normal. Consequently, one computes the space dimension of quadrics containing the curve to be $(g-2)(g-3)/2$ where $g$ is its genus, which is one of the versions by which his result is now known. The theorem was important for subsequent works by Enriques \cite{En} and Babbage \cite{B}, characterizing canonical curves which are not intersections of quadrics, by Petri \cite{P} for a complete description of the canonical ideal in terms of its equations and by several authors in a similar vein, see for instance \cite{AS,M, S, Sh,SV}.

\medskip
From \cite[pp.~113-117]{ACGH}, one sees that the result is a consequence of the canonical curves being extremal, i.e., they reach the Castelnuovo's bound, and these are projectively normal. So it's actually possible to extend the argument \emph{verbatin} to Gorenstein singular curves, those for which the dualizing sheaf is a bundle, and a canonical morphism well defined.

\medskip
For arbitrary integral curves, the problem is more subtle as the very notion of a canonical curve seems not to be so precise. In \cite{R}, Rosenlicht generalizes this concept as follows: for a curve $C$ with normalization $\cb$, its \emph{canonical model} $C'$ is the image of the morphism $\cb\rightarrow \mathbb{P}^{g-1}$ induced by the global sections of its dualizing sheaf. He proves that $C'\cong C$ if and only if the latter is non-hyperelliptic and Gorenstein, a result reobtained later on by distinct authors in many different ways, see for instance, \cite{C,F,H,MSD,SAk}. 

\medskip
In \cite{KM}, S. L. Kleiman, along with the second named author proved that if $C$ is non-Gorenstein, then $C'$ is projectively normal if and only if it is extremal, and characterized curves which are so, called \emph{nearly Gorenstein}. 
Highly based on this work, in \cite{Mt}, Noether's result is extended to those curves. In the same article, the general problem is addressed purely intrinsically, and the result proved for curves whose non-Gorenstein points are unibranch at worst. Later on, Feital in \cite{A}, and along with Contiero and the second named author in \cite{CFM}, led the study to an arbitrary number of branches and proved Noether's statement in bibranch case.

\medskip
This is the point of departure of this work. In fact, in \cite{Mt}, Noether's assertion was proved up to a combinatorial local result, 
\cite[Lem.~3.2]{Mt} (corresponding to our Lemma \ref{Lem2}), obtained just for unibranch singularities.  In \cite{CFM}, this local result was rephrased and proved in the bibranch case \cite[Lem.~3.2]{CFM} by means of a new combinatorial result for two-dimensional semigroups \cite[Lem.~ 3.1]{CFM}. The strategy here is essentially skip this step and address the combinatorial problem within another approach, which happens to be somewhat simpler and more general. Indeed, it works for an arbitrary number of branches, and for a general singularity, which need not to be non-Gorenstein, as in \cite{CFM,Mt}. So we get the following result, proved in Section \ref{CapAI}, which is Noether's statement in its full generality.

\begin{remark}\label{Teo2}
Let $C$ be a nonhyperelliptic integral and complete curve over an algebraic closed ground field, with dualizing sheaf $\omega$. Then maps
\begin{equation}\label{eq1}
\mbox{Sym}^{n}H^0(C,\omega)\longrightarrow H^0(C,\omega^n)
\end{equation}
are surjective for every $n\geq 1$.
\end{remark}

Afterwards, in Section \ref{S4},  
we complete equivalences already stated in \cite{CFM} which connect Noether's assertion with Rosenlicht's result and Green's conjecture at level zero, and which happen to be a multivalent characterization of non-hyperelliticity. Finally, in Section \ref{Approach}, we discuss another possible via to get Noether's statement linking the problem with a local theory developed by V. Barucci, M. D'Anna and R. Fr\"oberg in \cite{BDF,BDF2,BF}.

\medskip
\noindent{\bf Acknowledgements.} This work is part of the first named authour Ph.D thesis. The second named authour specially thank V. Barucci and G. Zito for some very helpful discussions when this article was yet to begin.

\section{Setup}
\label{prelim}

In this section, we introduce some notions  used throughout this text. As the present work is strongly based on \cite{CFM,KM,Mt}, we  follow the same notation, so that everything can be easily  carried over from those articles to this one.
 
\subsection{Linear Systems and Canonical Models}\label{SVsubsection1}

\medskip
For the remainder, $C$ is an integral curve,  complete over an algebraically closed ground field $k$, of arithmetic genus $g$, with structure sheaf $\oo_C$, or simply $\oo$. 

\medskip
A \emph{linear system of dimension $r$ and degree $d$ on $C$} is a formal pair $
 \sys =\sys(\fff ,V)$, where $\fff$ is a torsion free sheaf of rank $1$ on $C$, $V$ is a vector subspace of 
 $H^{0}(\fff )$ with dimension $r+1$ and $d:=\deg \fff :=\chi (\fff )-\chi (\oo)$, where $\chi$ is the Euler characteristic. The notation $g_{d}^{r}$ stands for a linear system of degree $d$ and dimension $r$. The linear system is said to be \emph{complete} if $V=H^0(\fff)$, in this case one simply writes $\sys=|\fff|$.
 When $\fff$ is a line bundle, then $\sys$ determines a morphism $C\to\mathbb{P}^d$. Otherwise, $\sys$ defines a pencil with  \emph{non-removable} base points, in the sense of \cite{Cp,RSt}. 

\medskip
A curve $C$ is said to be \emph{hyperelliptic} if there is a morphism $C\rightarrow\pp^1$ of degree $2$ or, equivalently, if it carries a base point free $g_2^1$. Now let $\ww$ be dualizing sheaf of $C$, recall that a point $P\in C$ is \emph{Gorenstein} if the stalk $\ww_P$ is a free $\oo_P$-module. The curve $C$ is said to be \emph{Gorenstein} if all of its points are so, or equivalently, if $\ww$ is invertible. When so, the series $|\ww|$ induces a morphism $\kappa:C\to\mathbb{P}^{g-1}$. If, additionally, $C$ is nonhyperelliptic then $\kappa(C)$ is often referred as a \emph{canonical curve}. 

\medskip
The latter notion can be extended as we will see right away. Before that, we just recall some helpful notation. Given a sheaf $\mathscr{G}$ on $C$, if $\varphi :X\to C$ is a morphism from a scheme $X$ to $C$, one sets
$$
\oo_{X}\mathscr{G}:=\varphi^* \mathscr{G}/\rm{Torsion}(\varphi^*\mathscr{G})
$$
and for each coherent sheaf $\fff$ on $C$ one sets 
$$
\fff^n:=\rm Sym^n\fff/\rm Torsion (\rm {Sym}^n\fff).
$$
In particular, 
if $\fff$ is invertible then clearly $\fff^{n}=\fff^{\otimes n}$.

\medskip
Now let $\cb$ be the normalization of $C$ and $\pi :\cb\rightarrow C$ the natural map. The linear
system $\sys(\oo_{\cb}\ww,H^0(\ww))$ induces a morphism $\psi :\cb\rightarrow{\mathbb{P}}^{g-1}$. Its image $C':=\psi(\cb)$ is the \emph{canonical model} of $C$, introduced by Rosenlicht. He proved in \cite[Thm.\,17]{R} that if $C$ is nonhyperelliptic, then there is a birational morphism 
$$
\pi' : C'\rightarrow C
$$
of which existence we will refer in Section \ref{S4} as \emph{Rosenlicht's statement}. In \cite{KM} one finds another characterization of $C'$ as follows. Let $\widehat{C}:=\text{Proj}(\oplus\,\ww ^n)$ the \emph{canonical blowup} of $C$, that is, the blowup of $C$ along its dualizing sheaf $\ww$. Rosenlicht's result easily implies the existence of a birational morphism $\widehat{C} \to C'$, but as $\ww$ is generated by global sections, then we actually have $C' \cong \widehat{C}$. 

\medskip
Now we study the maps above from a local perspective. First, let  $\widehat{\oo}$, $\oo'$ and $\overline{\oo}$ be the direct images on $C$ of the structure sheaves of $\cb$, $\widehat{C}$ and $ C'$ respectively, $\ccc:=\mathcal{H}\rm {om}(\overline{\oo},\oo)$
the conductor of $\overline{\oo}$ into $\oo$, and set $\overline{\oo}\ww:=\overline{\pi}_*(\oo_{\overline{C}}\ww)$. Given $P\in C$, take $\lambda\in H^0(\ww)$ such that $(\overline{\oo}\ww)_P=\overline{\oo}_P\lambda$ as in \cite[Lem.~2.8]{KM}. 
Let
$$
\vv := \ww/\lambda
$$
be the embedding of $\ww$ on the constant sheaf of rational functions of $C$ by means of $\lambda$.
By \cite[Lem.~6.1]{KM} we have the following sequence of inclusions
\begin{equation}\label{equooo}
\ccc_P\subset
\mathcal{O}_P \subset \vv_p\subset\oph=\op'\subset\obp.
\end{equation}

\subsection{Semigroups}\label{SVsubsection2}
Let $k(C)$ be the field of rational functions of $C$. Given $P\in C$, let $\pb_1,\ldots,\pb_s$ be the points of $\overline{C}$ over $P$, and $v_{\pb_1},\ldots, v_{\pb_s}$ the corresponding valuations. For any $x\in k(C)^*$, set
$$
\upsilon_{P}(x)=\upsilon(x):=(v_{\pb_{1}}(x),\cdots ,v_{\pb_{s}}(x))\in{\mathbb{Z}}^{s}.
$$
The \emph{semigroup of values} of $P$ is
$$
\sss := \upsilon(\op).
$$
Plainly, it is a subsemigroup of $\mathbb{N}^s$. Now for an arbitrary $a\in \mathbb{Z}^s$, let $a_i$ be its $i$-th component. 
We have that $\sss$ has the following additional properties
\begin{itemize}
\item[(i)] if $a, b \in \rm S$ then $ \min (a,b):=(\min(a_1,b_1),\ldots,\min(a_s,b_s)) \in \rm S$;

\item[(ii)] if $a, b \in \rm S$ and $a_i=b_i$ then there exists $\varepsilon \in \rm S$ such that
$\varepsilon_i > a_i=b_i$ and $\varepsilon_j \geq {\rm min}(a_j,b_j)$ (and equality holds if $a_j\neq b_j$);

\item[(iii)] there is $\beta\in\mathbb{N}^s$ such that $\beta+\mathbb{N}^s\subseteq \sss$.
\end{itemize}

\medskip
Properties (i)-(iii) make $\sss$ a \emph{good semigroup} in the sense of \cite{BDF}. It is also a \emph{local semigroup} according to the same reference. As $0$ is the only element in $\sss$ with a vanishing component, the term makes sense. Indeed, one may easily extend the above notion of semigroup of values to a semilocal ring $R$ with $k \subset R \subset k(C)$ just considering the $\pb_i$'s as the points of $\overline{C}$ such that $R\subset \oo_{\pb_i}$, $\forall i$. Then $R$ is local if and only if $v(R)$ is so.  

\medskip
For the remainder, we set  
$$
\alpha:={\min}({\rm S}\setminus\{ 0\})
$$
which agrees with the mutiplicity of $P$, that is $\alpha=\dim(\obp/\mmp\obp)$. On the other hand, the smallest element in $\mathbb{N}^s$ satisfying (iii) is the \emph{conductor} of $\sss$, denoted by $\beta$ as well. The relation with the conductor ideal is clear as 
$$
\beta={\min} (v(\cp))
$$ 
and the \emph{Frobenius vector} of $\rm S$ is $\gamma :=\beta -(1,\ldots ,1)$.

\medskip
A \emph{relative ideal} is any ${\rm E}\subset\sss$ such that $\sss+{\rm E}\subset {\rm E}$ and (iii) holds for ${\rm E}$. If, moreover (i) and (ii) hold as well, then ${\rm E}$ is said a \emph{good} relative ideal. The point here is that for any $\op$-fractional ideal, its set of values is a good relative ideal of $\sss$. 

Now given $E$, one sets
$$\Delta^{{\rm E}}(a):=\{b\in {\rm E}\mid b_i=a_i\hspace{0.4em}\mbox{for some } i\hspace{0.4em} \mbox{ and }\hspace{0.4em} b_j>a_j \mbox{ if } j\neq i\}.$$
The \textit{canonical ideal} of $\sss$ is
\begin{equation}\label{eqK}
\kk:=\{ a\in{\mathbb{Z}} ^{s}\ |\ \Delta ^{\rm S}(\gamma -a)=\emptyset\}.
\end{equation}
It is a good relative ideal of $\sss$ and plays a crucial role here. The term makes sense as well as we will see later on (cf. Equation \ref{eqk2}).

Finally, we fix a notation used throughout this text. Given $a:=(a_1,\ldots,a_s)\in{\mathbb{Z}} ^{s}$ we denote
$$|a|:=a_1+\ldots+a_s.$$
Now take local parameters $t_i$ for the $\oo_{\pb_i}$ and set
\begin{equation}\label{eqq1}
 t^a:=t_1^{a_1}\ldots t_s^{a_s}.   
\end{equation}
For instance, if so, one may write the conductor ideal as $\mathcal{C}_P=t^{\beta}\overline{\mathcal{O}}_P$.

\section{Proof of the Theorem}\label{CapAI}

In this section we prove Max Noether's statement for an arbitrary integral curve, that is, we get the result stated in the Introduction. As we said above,  \cite{KM}, \cite{Mt} and \cite{CFM} lead to partial generalizations of Noether's property. The main tool was the study of how global sections of the dualizing sheaf interact with each singularity, which gets harder as the number of branches grows.  In this sense, \cite{Mt} solves the problem extrinsically and when singularities are unibranch, while \cite{CFM} extends the analysis to the bibranch case. Thus all the global (and even part of the local) analysis are already contained in those works. 

\medskip
So the work here will be essentially a task of extracting a small part of \cite{CFM} and grafting in a new one. We recall all results needed from the prior papers, and develop the proof within the following steps.

\medskip
\noindent {\bf Step 1: Global Part.} First off, follow the whole proof of \cite[Thm. 3.7]{Mt}, which is Noether's statement allowing unibranch non-Gorenstein singularities at worst. It makes use of global arguments, most of them owing to Riemann-Roch, where the number of branches has no special role. But we have two things to adjust: (a) the powers $t_P^{\epsilon}$, with $\epsilon\in\zz$ make no sense if $P$ is not unibranch; (b) the proof relies in four other results, namely \cite[Lems. 3.2 to 3.6]{Mt}.

\medskip
To solve (a), first, one just have to allow $\epsilon$ be a vector in $\zz^s$ and make $t_P^{\epsilon}$ be as in (\ref{eqq1}); also, \cite[Lem.~3.2]{Mt}, which refers to this notation will be rephrased right below in Lemma \ref{Lem2} (as well as it was done in \cite[Lemma 3.2]{CFM}).

To start solving (b), first note that \cite[Lems. 3.3 and 3.4]{Mt} hold no matter the number of branches is. On the other hand, \cite[Lem.~3.6]{Mt} could perfectly had been stated for a multibranch $P\in C$, as long as it is at least triple. Indeed, in the proof, the point $\pb\in\cb$ over $P$ could've been replaced by any two points lying over $P$, in case the latter is not unibranch; and for the rest, one deals with $|a|$, rather than $a\in\zz^s$, noticing that the module is an additive function. 

To finish (b), note that the unibranch assumption in \cite[Lem. 3.5]{Mt} is needed only at the last paragraph. So choose any $\pb\in\cb$ over $P$ and force the function $a$ appearing at the end of the proof of Lemma \cite[Lem. 3.3]{Mt} to have a zero on it. One may find the desired differentials, replacing $v_P(y_i)$ by $|v_P(y_i)|$. Now \cite[Lem. 3.2.(iii)]{Mt} -- which will be addressed right away -- takes care of the rest of the proof.

\medskip
\noindent {\bf Step 2: Local Part.} As said above, our remaining (and main) task is to adjust (and prove) \cite[Lem.~3.2]{Mt}, which we rephrase here as the following result. The statement is essentially \cite[Lem.~3.2]{CFM} up to the fact that we make no assumption that the point is bibranch, at worse. The reason why, is that we will skip \cite[Lem.~3.1]{CFM} by means of a simpler and more general argument.

\begin{lem}\label{Lem2}
Let $P\in C$ be a singular point, then the maps
$$\begin{array}{lrr}
H^0(\vv)^n & \longrightarrow & \dfrac{\vv_P^n}{t_{P}^{-\epsilon}\mathcal{C}^n_P}
\end{array} $$
are surjective for every $n\geq 2$ in the following cases:
\begin{enumerate}[(i)]
\item $|\epsilon| =2n-1$ and $\epsilon \leq (n-1)\alpha$;
\item $|\epsilon|=1$ if there is $f_0\in H^0(\vv)$ with $\upsilon_P(f_0)=\beta$;
\item $ |\epsilon|=0$ if there are $f_0, f_1 \in H^0(\vv)$ with $\upsilon_P(f_0)=\beta$ and $\upsilon_P(f_1)=\beta+u$ with $|u|=1$ or $2$.
\end{enumerate}
\end{lem}

To get the result, we will follow the proof of \cite[Lem.~3.2]{CFM}. Note that the strategy is first proving the result in case (i) and $n=2$ and then easily get it for higher $n$ and all cases. By its turn, even this task is subdivided. Indeed, consider the following sequence of three inclusions
\begin{equation}\label{eqs1}
    \vv_P^2\supset\cp \supset t^{\beta-\alpha}_P\cp= t^{-\alpha}_P\cp^2\supset t_P^{-\epsilon}\cp^2.
\end{equation}
each of which corresponds to a part of the proof.

First, recall the following equation which can be found in \cite[p. 117 bot]{St}, reread in \cite[Lem.\,6.1]{KM}
\begin{equation}
\label{equwhc}
\vv_P=H^0(\vv)+\cp.
\end{equation}
This yields 
$\vv_P^n= H^0(\vv)^n+\cp$
and hence epimorphisms 
$
H^0(\vv)^n\twoheadrightarrow\vv_P^n/\cp.
$
In particular, we have the surjective linear morphism
\begin{equation}
\label{equst1}
H^0(\vv)^2\twoheadrightarrow\vv_P^2/\cp.
\end{equation}

So we have to deal with the other two inclusions of (\ref{eqs1}). The main task is proving that
\begin{equation}
\label{equcbn}
\cp /t^{\beta-\alpha}\cp=\overline{H^0(\vv)^2\cap\cp}.
\end{equation}
where the right hand side denotes the set of classes of elements of $H^0(\vv)^2\cap\cp$ mod $t^{\beta-\alpha}\cp$. This is the core of the whole argument, the proof of which we will do right away (in next step). For now, we just observe that the remaining tasks, that is, the proof of 
\begin{equation}
\label{equcbn2}
t^{-\alpha}\cp^2 /t^{-\epsilon}\cp^2=\overline{H^0(\vv)^2\cap t^{-\alpha}\cp^2}.
\end{equation}
and how one gets the lemma for any $n$ and all itens (i)-(iii) out of the case $n=2$, is something the reader can follow verbatim \cite{CFM}, from page 126 on.

\medskip
\noindent {\bf Step 3: Combinatorial Part.} From what was said above, we just have to prove (\ref{equcbn}), and we will be done. 
We start by rephrase it into a pure combinatorial statement. The key fact allowing us to do so is that, by \cite[Prp. 2.14.(iv)]{BDF}, we have
\begin{equation}\label{eqk2}
    v_P(\vv_P)=\kk
\end{equation}
the canonical ideal defined in (\ref{eqK}). So we claim that (\ref{equcbn}) reduces to the existence of a sequence
\begin{equation}\label{eqc1}
a_1=\beta<a_2<a_3<\ldots< a_{|\beta-\alpha|}<2\beta-\alpha
\end{equation}
such that
$a_i\in \kk^{\circ}+\kk^{\circ}$, where ${\kk}^{\circ}:=\{a\in {\kk}\, |\, a<\beta\}$. Indeed, this implies the existence of elements $f_i\in H^0(\vv)^2$ such that $v_P(f_i)=a_i$, owing to (\ref{equcbn}), which are in $\cp$, since $v_P(f_i)\geq \beta$, and all linear independent mod $t^{\beta-\alpha}\cp$, because the sequence is strictly increasing. Moreover, the amount of the $f_i$ is $|\beta-\alpha|=\dim(\cp /t^{\beta-\alpha}\cp)$ and the claim follows.

In order to find the desired sequence (\ref{eqc1}), we first fix some useful notation:
\begin{align*}
\alpha^{i} &:={\mathrm min}(i\alpha, \beta); \\
m &:= \text{largest integer such that }(m+1)\alpha\leq\beta;\  \\
& \,= \text{largest integer such that }\ \alpha^{m+1}=(m+1)\alpha; \\
r&:=\text{smallest integer such that}\ (r+2)\alpha>\beta
\end{align*}
and let also $\{e_1,\ldots,e_s\}$ be the standard basis for $\mathbb{N}^s$. Now split (\ref{eqc1}) into three parts: (I) from $\beta$ to $\beta+m\alpha$; (II) from there to $\beta-\alpha+\alpha^{r+1}$; (III) and then to $2\beta-\alpha$, as in the figure below. 

\begin{figure}[!htb]
\centering
\includegraphics[width=10cm]{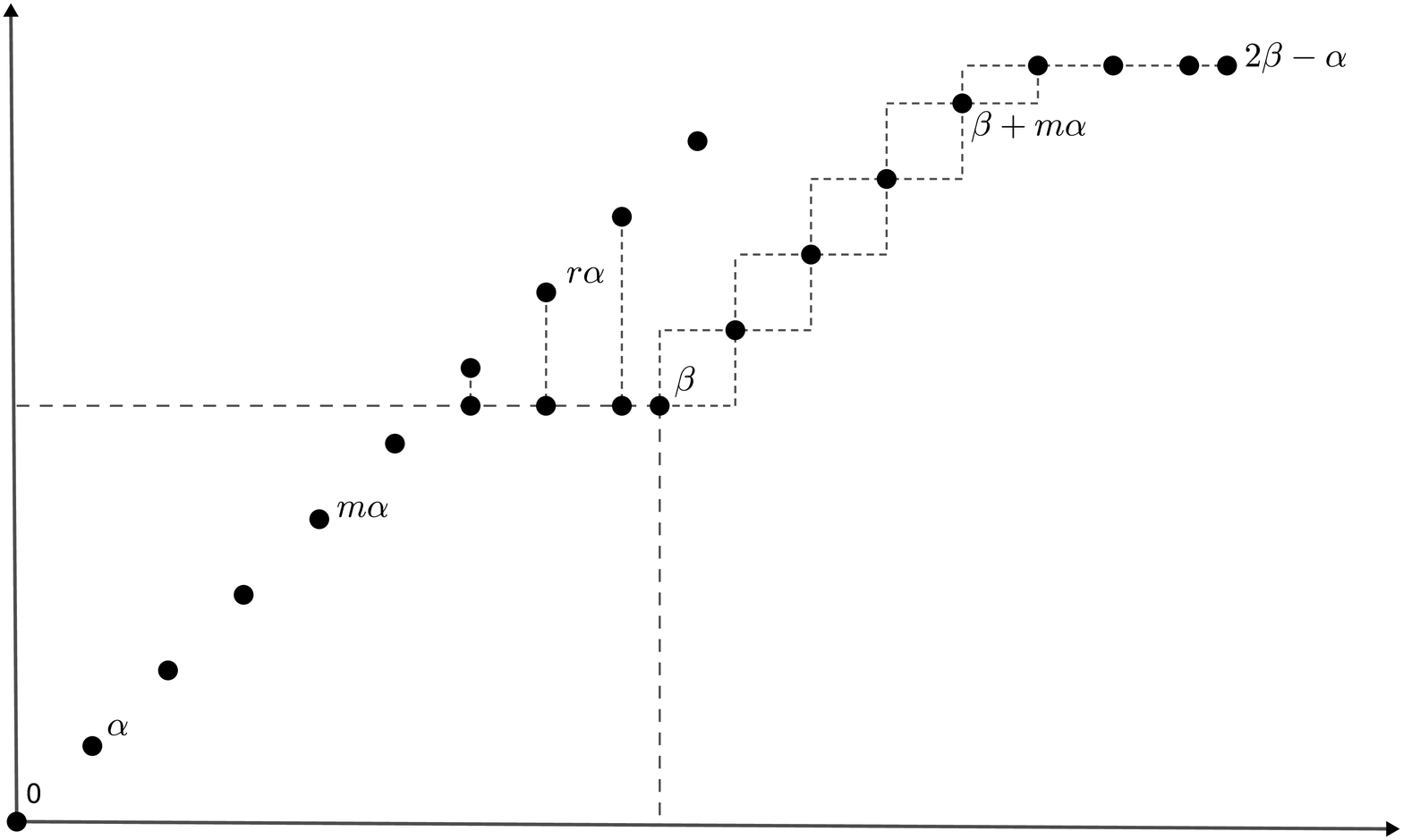}
\end{figure}

\pagebreak

And now note that the argument used in \cite[p. 125]{CFM} to build the sequences for parts (II) and (III) above, do not depend on the number of branches. So we are reduced to prove the result below to finally get Max Noether. As we said, it stands for a shortcut, simpler and more general, to \cite[Lem.~3.1]{CFM}.

\medskip

\begin{lem}\label{maismaroto} There exists a sequence
\begin{equation}
\label{seq1}
\beta=c_0<c_1<\ldots<c_i<\ldots<c_{m|\alpha|-1} < \beta+m\alpha
\end{equation}
such that $c_i \in \kk^{\circ}+\kk^{\circ}$ for every $i\in\{0,\ldots,m|\alpha|-1\}$.

\end{lem}

\begin{proof}
For every $i\in \nn^*$ set
$$
R_i:=\{a\in \mathbb{N}^s\mid (i-1)\alpha<a<i\alpha\}
$$
and let $n$ be the smallest integer such that $R_n\cap\sss\neq\emptyset$.

\medskip
By the definition of ${\mathrm K}$ it is immediate to note the following inclusions
\begin{equation}
\label{equaaa}
A_l:=\{\beta-l\alpha+b ~ | ~ 0 \leq b <\alpha\ {\mathrm{with}}\ |\alpha - b|\geq 2\}\subset{\mathrm K}^{\circ}
\end{equation}
for every $1\leq l\leq n-1$.

\begin{figure}[h!]
	
	\centering
	\includegraphics[width=7cm]{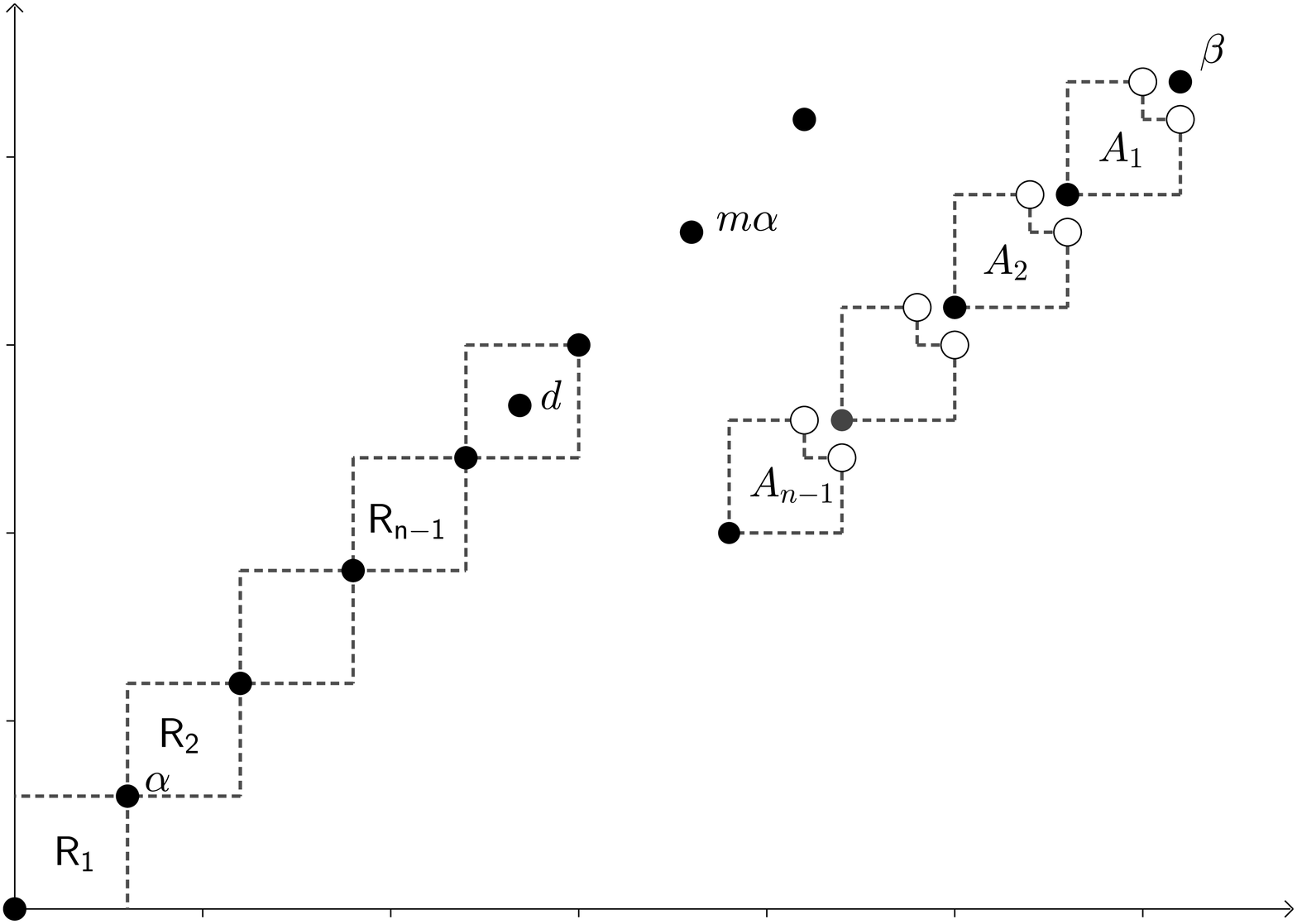}
\end{figure}

So we have 
$$
A_1+j\alpha \subset \kk^{\circ}+\kk^{\circ}\ \ \text{for}\ \ 1\leq j \leq m.
$$
Therefore, every sequence like (\ref{seq1}) such that $c_{i|\alpha|}=\beta+i\alpha$ is contained in $\kk^{\circ}+\kk^{\circ}$, up to the points $c_{i|\alpha|-1}$, which are the form $\beta+i\alpha-e_{j}$ for some $j$.

\medskip
Take $d\in R_n\cap \sss$. We have 
\begin{equation}
\label{equddk}
d+k\alpha \subset \kk^{\circ}
\end{equation}
for $0\leq k\leq m-n+1$ and we also have
\begin{equation}
\label{equmar}
\beta+i\alpha-e_{j} \in A_l+d+k\alpha\ \ \ \mbox{ for } i=n-l+k\ \text{and some}\ j .
\end{equation}
In fact, write an element in $A_l + d + k \alpha $ as $ \beta + (n-l + k-1) \alpha + q + b $ where $ d = (n-1) \alpha + q$; then take $b := (\alpha-u) -e_j$ such that $b\geq 0$ and the claim follows.

\medskip
To find the desired $c_{i|\alpha|-1}$'s, by (\ref{equaaa}), (\ref{equddk}) and (\ref{equmar}), it suffices to show that for all $i\in\{1,\ldots,m\}$ we can find $k\in\{0,\ldots,m-n+1\}$ and $\ell\in\{1,\ldots,n -1\}$ such that $i=n-l+k$. But this is immediate from the very variation of these indices. Thus take $c_{i|\alpha|-1}=\beta+i\alpha-e_j$ for each $1\leq i\leq m$ and the result follows.
\end{proof}

\

\

\section{Characterizing Non-Hyperelliptic Curves}\label{S4}

In this section we connect Noether, Rosenlicht and Green statements. The result we prove stands for a characerization of non-hyperellipticity. It appears in \cite[Thm. 4.1]{CFM}, but part of the equivalences relies on bibranch hypothesis. In order to state it, we recall a few concepts.

\medskip
Let $\fff$ be a torsion-free sheaf of rank 1 on $C$. Following \cite[p. 363 Dfn. 2.2 (7)]{Bal}, the \emph{Clifford index} of $C$ is
$$
\cliff(C)=\min\{\deg\,\fff-2(h^0(\fff)-1)\, ;\, h^0(\fff)\geq 2\ {\rm{e }}\ h^1(\fff)\geq 2\}.
$$
Now following \cite{ApF, Gr}, consider the complex
\begin{equation}\label{eqs}
\wedge^{p+1}H^0(\fff)\otimes H^0(\fff^{q-1})\stackrel{\phi_{p,q}^1}{\longrightarrow}\wedge^ {p}H^0(\fff)\otimes H^0(\fff^{q})\stackrel{\phi_{p,q}^2}{\longrightarrow} \wedge^{p-1}H^ 0(\fff)\otimes H^0(\fff^{q+1});
\end{equation}
the quotient
$$
K_{p,q}(C,\fff):= \ker(\phi_{p,q}^2)/{\rm im}(\phi_{p,q}^1)
$$
is the \emph{$(p,q)$-th Koszul cohomology} of $\fff$. The approach here, as in \cite{CFM}, allows torsion free sheaves on both definitions, and not just bundles as in the mentioned refrences.

\medskip
Now \emph{Green's conjecture} links both concepts in the following statement \cite{Gr}:
$$K_{p,2}(C,\omega)=0 \Longleftrightarrow p < \cliff(C).$$
It was proved for smooth curves in general by C. Voisin in \cite{Vo1, Vo2} while in \cite{BFT,FT}, e.g., one can find such a study for classes of possibly singular curves. It is well known that, for regular curves, the conjecture at level $p=0$ is essentially Max Noether statement, while level $p=1$ corresponds to Enriques-Babbage \cite{En,B} celebrated characterization of trigonal curves.

\medskip
Now let us carefully look at the $p=0$ case, assuming the curve might be singular. Taking $q=2$ and $\fff=\ww$ in (\ref{eqs}) we get
$$
H^0(\ww)\otimes H^0(\ww)\stackrel{\phi_{0,2}^1}{\longrightarrow} H^0(\ww^2)\longrightarrow 0
$$
Thus $K_{0,2}=0$ is equivalent to saying that $\phi_{0,2}^1$ is surjective, that is, Noether holds for $n=2$. But we just saw in the previous section that once Noether's statement holds for $n=2$ one easily derives all other surjections. On the other hand, in \cite[App]{EKS} it is proved that $\cliff(C)=0$ if and only if $C$ is either hyperelliptic or rational with a unique singularity whose maximal ideal agrees with the conductor. Curves with such a property appear in the literature in many different ways. They are called \emph{curves defined by a module} in Serre's \cite{Serre1959},  \emph{partitional} in Behnke-Christophersen 
\cite{BC}, \emph{universal} in Steven \cite{St}, and \emph{nearly normal} in Kleiman-Martins \cite{KM}, the term we adopt below, and where is given a characterization of which in the rational case. So these curves happen to be the key ingredient to adjust Green's conjecture to the integral case, as stated in the last two items of the following result.

\begin{thm}\label{ap1equiv}
The following are equivalent:
\begin{itemize}
\item[(i)] $C$ is non-hyperelliptic;
\item[(ii)] Rosenlicht's statement holds: there is a birational morphism $C'\to C$;
\item[(iii)] Noether's statement holds: ${\rm Sym}^nH^0(\ww)\to H^0(\ww^n)$  is surjective $\forall n\in\nn$;
\item[(iv)] $K_{0.2}(C,\ww)=0$
\item[(v)] $\cliff(C)>0$ or $C$ is rational nearly normal.
\end{itemize}
\end{thm}.

The previous discussion yields the equivalences (iii)$\Leftrightarrow$(iv) and (iv)$\Leftrightarrow$(v). Now if $C$ is hyperelliptic then $C'$ is the rational normal curve of degree $g-1$ in $\mathbb{P}^{g-1}$, and in particular is not birationally equivalent to $C$, thus (ii)$\Rightarrow$(i); on the other hand, (i)$\Rightarrow$(ii) is Rosenlicht's Theorem \cite[Thm. 17]{R}, so (i)$\Leftrightarrow$(ii). Also, if $C$ is hyperelliptic, then $\dim(H^0(\ww)^2)=2g-1$ owing to \cite[p.~95 bot]{KO1}, while $h^0(\ww^2)=3g-3$ due to \cite[Lem.~3.4]{Mt} and Riemann-Roch, thus (iii)$\Rightarrow$(i); now (i)$\Rightarrow$(iii) is our main theorem, so (i)$\Leftrightarrow$(iii), and we are done with the proof of the above result. We just recall here \cite[Rmk.~2.10]{Mt} where is discussed a way of getting (ii) by means of (iii). So since we got Noether's statement in this work, we also get an alternative way of proving Rosenlicht's statement.

\section{Another Local Approach}\label{Approach}

In this section we briefly discuss a different way of addressing Noether's result. It connects the subject with some local commutative algebra developed by Barucci, D'Anna and Fr\"oberg in \cite{BF,BDF,BDF2}. At the end we get the theorem for curves with bibranch non-Gorenstein points at worse, as in \cite{CFM}; to go further, one needs \cite[Thm.~4.6, Cor.~4.7]{BDF2} for an arbitrary number of branches. The gain is linking the problem with some alternative local theory in a way that Noether's assertion can be derived from an independent local statement. Also, this approach congregates both intrisic and extrinsic points of view. Indeed, the extrinsic proof of Noether in \cite[Thm.~2.7]{Mt} is used here to lead us to a local, and hence intrinsic, problem of rings named almost Gorenstein, as we will see.

\medskip
We start by adjusting \cite{BDF2} to our framework. The reader can check that all results mentioned here work perfectly to any local ring of a point in a curve.

\begin{defi}
\emph{Let $P\in C$ be any point. Given $I,J$ fractional ideals of $\op$, set 
$$
(I: J):=\{a\in k(C) \mid a J\subseteq I\}
$$
which agrees with the conductor $\cp$ when $I=\op$ and $J=\obp$. Following \cite[Chp.~4]{BDF2} $P$ is called
\textit{maximal with fixed conductor} if for every ring $T$ such that
$\op\varsubsetneq T \subset \obp
$, 
we have a strict inclusion $(\op:\obp) \subsetneq\, (T:\obp)$.}
\end{defi}

And now we recall a couple of local properties with a strong geometric appeal as we will see right away.

\begin{defi} \label{defnng}
\emph{Let $P\in C$ be any point. Set
$$
\eta_P:=\dim(\vlp/\op)\ \ \ \ \ \ \ \ \ \ \ \mu_P:=\dim({\widehat{\oo}_{P}}/\vlp)
$$
and also
$$
\eta:=\sum_{P\in C}\eta_P\ \ \ \ \ \ \ \ \ \ \mu:=\sum_{P\in C}\mu_P
$$
Following \cite[pp. 418, 433, Prps. 21, 28]{BF} call $P$ \emph{Kunz} if $\eta_P=1$ and,
accordingly, say $C$ is \emph{Kunz} if all of its non-Gorenstein points are Kunz. Following \cite[Dfn. 5.7]{KM} call $C$ \emph{nearly Gorenstein} if  $\mu=1$. 
}
\end{defi}

The key fact that links both definitions is \cite[Thm. 4.6, Cor. 4.7]{BDF2} from which we get that if $P$ is bibranch at worst, then it is maximal with fixed conductor if and only if it is Gorenstein or Kunz. Moreover, if $P$ is Kunz then it is almost Gorenstein owing to  \cite[Prp.~21]{BF}. On the other hand, \cite[Thm. 2.7]{Mt} provides a total geometric and extrinsic proof of Noether's statement to nearly Gorenstein curves. The reason why is that such curves are precisely the ones for which the canonical model $C'$ is projectively normal. 

\medskip
So, for the remainder, we will merge the facts above to get a third proof of Noether if the singularities are bibranch at worse. But the reader should note that we are actually proving something else, i.e., assume maximal with fixed conductor implies almost Gorenstein, then Noether holds.   

\medskip
Let $P\in C$ be a non-Gorenstein point, and $C^{\#}$ be the curve obtained by resolving all non-Gorenstein singularities of $C$ distinct from $P$. We have a birational morphism $\pi^{\#}:C^{\#}\to C$. Let $\omega^{\#}$ be the dualizing sheaf of $C^{\#}$. Set $\mathcal{O}^{\#}=\pi^{\#}_{\ast}(\mathcal{O}_{C^{\#}})$ and $\vv^{\#} :=\pi^{\#}_{\ast}( \ww^{\#})/\lambda$, 
where $\lambda$, as before, is such that $(\overline{\oo}\ww)_P=\overline{\oo}_P\lambda$. Thus 
\begin{equation}
\label{equhch}
H^0(\vv^\#)\subset H^0(\vv).
\end{equation}
Also, by \cite[Lems. 2.8 and 6.1]{KM}, we have that $\vv^{\#}_{P}$ is determined by the conductor of the local ring. But, in this case,  $\oo^{\#}_P=\oo_{P}$ by construction, and hence
\begin{equation}
\label{equwis}
\vv^{\#}_P=\vv_P.
\end{equation} 
So we have the following morphisms
\begin{equation}
\label{equsic}
H^0(\vv)^n \supset H^0(\vv^{\#})^{ n}\longrightarrow (\vv_P^{\#})^n\slash t^{-\epsilon}\mathcal{C}^n_P = \vv_P^n\slash t^{-\epsilon} \mathcal{C}^n_P
\end{equation}
where $\epsilon$ is as in the statement of Lemma \ref{Lem2}.

\medskip
So now we will prove this lemma by a different argument if $P$ is bibranch at worst. First off, from the proof of \cite[Thm. 3.7]{Mt} one sees that Lemma \ref{Lem2}, besides sufficient, is also necessary to get $H^0(\vv)^n=H^ 0(\vv^n)$. So assume $P$ is almost Gorenstein. Then $C^{\#}$ is nearly Gorenstein and satisfies Noether statement owing to \cite[Tmh.~2.9]{Mt}. Therefore, the maps in (\ref{equsic}) are surjective, which implies that the maps $H^0(\vv)^n \to \vv_P^n\slash t^{-\epsilon} \mathcal{C}^n_P$ are surjective as well.

\medskip
Now take any $P\in C$. If $P$ is maximal with fixed conductor, then it is almost Gorenstein and we are done, from what was said above.  Otherwise there is and intermmediate ring $\op^*$ such that
$$
\mathcal{O}_P\varsubsetneq \op^* \subset \overline{\mathcal{O}}_P
$$
and also
$$
\cp^*:=(\op^*: \overline{\mathcal{O}}_P)= (\mathcal{O}_P: \overline{\mathcal{O}}_P)=\cp.
$$
Then exclusively either:
\begin{itemize}
    \item[(i)] $\op^*$ is local and maximal with fixed conductor;
    \item[(ii)] $\op^*$ is not local;
    \item[(iii)] $\op^*$ is local, but not maximal with fixed conductor.
\end{itemize}

\medskip
If (i) holds, consider the partial dessingularization $\phi:C^*\to C$ such that: 
$\phi^{-1}(P)=\{P^*\}$ with $\phi_*(\oo_{C^*,P^*})=\op^*$ and such that $C^*\setminus\{P^*\}\cong C\setminus\{P\}$. Then $P^*$ is almost Gorenstein since $\op^*$ is maximal with fixed conductor; thus $C*$ is nearly Gorenstein, and we get the lemma for $C^*$ and $P^*$. But as $\cp^*=\cp$ we can apply exactly what was said above replacing $C^{\#}$ by $C^*$ to get the lemma for $C$ and $P$ as well. 

\medskip
If (ii) holds, then $\upsilon_P(\op ^*)$ has a nonzero element with a vanishing component. On the other hand, by \cite[Lem 4.1.]{BDF2} we have $\upsilon_P(\op^*)\subset \kk$; and, in particular, $\kk$ has such an element and this case is trivial as noted in the proof of \cite[Lem.~3.1]{CFM}. And if (iii) holds, consider $C^*$ and $P^*$ as defined in (i). Replace $C$ by $C^*$, $P$ by $P^*$ and restart the procedure, which naturally ends as $\dim(\overline{\mathcal{O}}_P/\mathcal{O}_P )<\infty$.


\end{document}